\documentclass[11pt]{article}

\usepackage[utf8]{inputenc}
\usepackage[T1]{fontenc}
\usepackage{lmodern}
\usepackage[english]{babel}
\newcommand{\og}{«\:}
\newcommand{\fg}{»}

\usepackage{amsmath}
\usepackage{amssymb}
\usepackage{amsfonts}
\usepackage{amsthm}

\usepackage{geometry}

\usepackage{cite}

\usepackage{hyperref}
\usepackage{graphicx}
\usepackage{svg}

\usepackage{algorithm}
\usepackage{algpseudocode} 

\theoremstyle{plain}
\theoremstyle{definition}
\newtheorem{definition}{Definition}
\newtheorem{theorem}[definition]{Theorem}
\newtheorem{corollary}[definition]{Corollary}
\newtheorem{notation}[definition]{Notation}
\newtheorem{proposition}[definition]{Proposition}
\newtheorem{remark}[definition]{Remark}
\newtheorem{hypothesis}[definition]{Hypothesis}

\hypersetup{
    colorlinks,
    citecolor=black,
    filecolor=black,
    linkcolor=black,
    urlcolor=black
}

\newgeometry{left = 1in, right = 1in, top = 1in}

\newcounter{cnt}
\title{\textbf{Optimal transport, determinantal point processes and the Bergman kernel}}
\author{William DRIOT \thanks{LTCI, Télécom Paris \\ 19, place Marguerite Perey, 91120 Palaiseau \\ Contact : william.driot[at]telecom-paris.fr, laurent.decreusefond[at]telecom-paris.fr } \and Laurent DECREUSEFOND }
\date{June 2025}

\begin{document}

\maketitle

\begin{abstract}
    We study the Bergman determinantal point process from a theoretical point of view motivated by its simulation. We construct restricted and restricted-truncated variants of the Bergman kernel and show optimal transport inequalities involving the Kantorovitch-Rubinstein Wasserstein distance to show to what extent it is fair to truncate the restriction of this point process to a compact ball of radius $1 - \varepsilon $. We also investigate the deviation of the number of points of the restricted Bergman determinantal point process, indicate which number of points looks like an optimal choice, and provide upper bounds on its deviation, providing an answer to an open question asked in \cite{Decreusefond2016}. We also consider restrictions to other regions and investigate the choice of such regions for restriction. Finally, we provide general results as to the deviation of the number of points of any determinantal point process.
\end{abstract}

\tableofcontents

\setlength{\parindent}{15pt} 
\setlength{\parskip}{6pt}    

\section{Introduction}

Introduced for the first time in \cite{Macchi1975}, determinantal point processes (DPPs) are a particular class of point process, the correlations functions of which can be written as determinants. After having been thoroughly studied from a probabilistic point of view, (see \cite{ShiraiTakahashi2003} and \cite{Soshnikov2000} for instance, among many others), they have been used to model fermion particles \cite{Tamura2006}. They have suprisingly and remarkably appeared in the study of two very elegant mathematical problems : the first one is the point process defined as the set of eigenvalues of random $ n \times n $ matrices with independent complex standard gaussian entries. It turns out that, as n goes to infinity, this point process converges in law to a point process that can be described \cite{Ginibre1965} as a determinantal point process, the kernel of which was thereafter called the Ginibre kernel. The second one is the point process defined as the set of zeroes of Gaussian Analytical Functions (GAF) in the unit disc of the complex plane - the former are analytical functions (i.e., power series), the coefficients of which are all independent complex standard gaussian random variables : as it turns out, this point process can also be described as a determinantal point process, the kernel of which is this time the Bergman kernel (see \cite{Hough2009ZerosGAF}, \cite{ZerosGaussianPowerSeries}). DPPs also appear and are used in the study of a vide variety of different mathematical problems, such as self-avoiding walks, random spanning trees or random integer partitions, see \cite{Burton1993}, \cite{Meliot2021}, \cite{Tao2024}. More recently, some peers attempted to apply them to machine learning \cite{DeterminantalML2013}, and to model phenomena arising in the field of networking \cite{Miyoshi2014}, \cite{TorrisiLeonardi2014}, \cite{Vergne2014}. The fact that these processes exhibit repulsion between particles is the core property that makes them very fitting for many applications.

The two aforementionned kernels, respectively Bergman and Ginibre kernels, benefit from several interesting properties. While the former, taking its values in the unit disc of the complex plane, is invariant with respect to rotations, the latter, which is valued in the whole complex plane, is invariant with respect to rotations and translations. Whether it be for quantum mechanics, machine learning or networking modelling, for pratical uses it is necessary to benefit from a pratical simulation algorithm. This work has been initiated in \cite{Hough2006} wherein the authors give a practical algorithm for the simulation of DPPs. The simulation procedure which is hinted in \cite{Ginibre1965} was fully developed in \cite{Caer1990}.

The main problem of these algorithms is that, strictly speaking, they unfeasible due to the fact that the point processes involved exhibit an infinite number of points almost surely, and that, in the case of the Ginibre kernel, the points scatter throughout the whole plane. In \cite{DecreusefondMoroz2021}, the authors construct restricted and trucated variants of the Ginibre point process, that allow simulation to become closer to feasible. More precisely, the Ginibre point process is therein restricted onto a ball of arbitrary radius and centered at the origin. The number of points then becomes almost surely finite. However, the stochastic dynamics of DPPs implies that computing the total number of points in the process requires to simulate a countable infinity of Bernoulli random variables, which is still unfeasible. The only option left is then to truncate the DPPs to a fixed (and finite) number of points, and brush off the remaining ones. But then, is it possible to quantify the approximation error that inevitably comes with this method ? To what extent is this new probability measure close
to the original one ? Theoretical results were shown in \cite{DecreusefondMoroz2021}, providing answers to this question. For instance, the authors have shown that denoting $\mathfrak S^R $ and $ \mathfrak S^R_N $ the Ginibre DPP restricted to a ball of radius $R$ centered at the origin of the complex plane and its truncation to $N$ eigenvalues respectively, then truncating this DPP to $N_R = (R + c)^2 $ yields, for all $ c > 0 $,
\begin{equation}\label{eq:main-result-ginibre}
    \mathcal W_{KR} (\mathfrak S^R, \mathfrak S^R_N ) \leqslant \sqrt{\frac 2 \pi } R e^{-c^2} 
\end{equation}
whenever $ R > c$, where $ \mathcal W_{KR} $ denotes the Kantorovitch-Rubinstein (or Wassertein-1) distance on the set of probability measures on the set of configurations on $\mathbf C$. This results shows that the (Wasserstein) distance between the two probability measures is exponentially small in $c^2$ as the number $ N_R = (R+c)^2 $ \textit{deviates} from $ R^2 $, and implies that the two DPPs $ \mathfrak S^R $ and $ \mathfrak S^R_{N_R} $ coincide with high probability. As a result, truncating to $R^2$ points looks like a very good choice, and this is why it is what they suggested (and implemented in \cite{MorozSoftware}). 

This is all very nice and convenient, but only applies to one DPP law, the Ginibre DPP. In this paper, we extend this work, by constructing restricted and restricted-trucated versions of the Bergman DPP, which will henceforth be at the very core of our study. We find a good compromise as to the right number of points depending on the chosen radius $R$ to truncate our restricted DPP and provide the theoretical results that ensure that the use of the simulation algorithm presented in \cite{DecreusefondMoroz2021} produces a result that is not blatantly far from the (theoretical and nonfeasible) simulation of the restricted Bergman DPP. In other words, the truncation error is small when choosing the right number of points, hence providing an answer to the fifth open question asked in \cite{Decreusefond2016} for this DPP. We also provide some general results for DPPs that are not specific to the Bergman kernel.

\begin{figure}[h]\label{fig:main-fig}
    \centering
    \begin{minipage}{0.45\textwidth}
        \centering
        \includegraphics[width=\linewidth]{bergman.png} 
        \caption{Simulation of the Bergman DPP restricted to a radius of 0.9995, using \cite{MorozSoftware}. Number of points : 985.}
    \end{minipage}
    \hfill
    \begin{minipage}{0.45\textwidth}
        \centering
        \includegraphics[width=\linewidth]{bergman_zoom.png} 
        \caption{A zoom on the boundary of the simulation on the left.}
    \end{minipage}
\end{figure}

\section{Determinantal Point Processes}

Let $E$ be a Polish space, equipped with the Borel $\sigma$-algebra generated by the open subsets of $E$. In the sequel, $\lambda$ is a reference Radon measure on this measurable space. Denote $\mathfrak{N}$ the space of locally finite subsets in $E$, also called the configuration space : $$ \mathfrak{N} = \left\{ \xi \subset E \mid \forall \Lambda \subset E \text{ compact},\, |\Lambda \cap \xi| < \infty \right\}. $$ 
It is equipped with the topology of vague convergence, also called weak topology, which is the weakest topology such that for all continuous and compactly supported functions $f$ on $E$, the mapping $ \xi \mapsto \langle f, \xi \rangle := \sum\limits_{x \in \xi} f(x) $ is continuous.  Once topologized, it naturally comes with a $\sigma$-algebra. Elements of $ \mathfrak N$, i.e., locally finite configurations on $E$, are identified with atomic Radon measures on $E$.

Next, considering a compact subset $\Lambda \subset E$, we may consider the set $\mathfrak{N}_\Lambda = \{\xi \subset \Lambda,\, |\xi| < \infty\}$ of finite configurations on $\Lambda$, equipped with the trace $\sigma$-algebra.

A point process is then any random variable valued in $\mathfrak{N}$.

\begin{definition}
    Let $\eta$ be a point process on $E$. For $n \geq 1$, the $n$-th factorial moment of $\eta$ is the point process defined on $E^n$ as the set of $n$-tuples of points of $\eta$ :
    \[ 
        M_\eta^n = \{(x_1, \dots, x_n) \mid x_1, \dots, x_n \in \eta\}.
    \]
\end{definition}

\begin{definition}
    The $n$-th factorial moment measure of $\eta$ is the measure defined by
    \[ 
        \mu_\eta^n(B_1 \times \dots \times B_n) = \mathbb{E}(M_\eta^n(B_1 \times \dots \times B_n)),
    \]
    where $B_1, \dots, B_n$ are measurable subsets of $E$. This relation defines a unique measure on $E^n$ equipped with the product $\sigma$-algebra.
\end{definition}

\begin{definition}
    If $\mu_\eta^n$ admits a Radon-Nikodym derivative $\rho_n$ with respect to $\lambda^{\otimes n}$, the latter is called the $n$-th correlation function of $\mu$. In other words, we have, assuming their existence
    \[ 
        \rho_n(x_1, \dots, x_n) = \frac{\mathrm d (\mu_\eta^n )}{\mathrm d (\lambda^{\otimes n})}(x_1, \dots, x_n).
    \]
    Equivalently, a point process $\eta$ is said to have correlation functions $(\rho_n)_{n \geq 1}$ if for all disjoint bounded Borel subsets $A_1, \dots, A_n$ of $E$, we have
    \[
        \mathbf E \left[ \prod_{k=1}^{n} \xi(A_k) \right] = \int_{A_1 \times \dots \times A_n} \rho_n(x_1, \dots, x_n) \mathrm d\lambda^{\otimes n}.
    \]
\end{definition}

Correlation functions are symmetric and characterize processes. $\rho_1$ is interpreted as the average density of particles with respect to $\lambda$. More generally,
\[
    \rho (x_1, ..., x_n) \mathrm d \lambda (x_1) ... \mathrm d \lambda (x_n),
\]
represents the probability of the event \og For each $ k \in [ \! [1, n] \! ] $, there is a (at least one) point at the
vicitiny $ \mathrm d x_k $ of $ x_k $ \fg. See \cite{PPIntro2003} for an introduction to this topic.

For any compact set $ \lambda \subset E $, we denote by $ L^2(\Lambda, \lambda) $ the Hilbert space the Hilbert space of complex-valued square integrable functions with respect to the restriction of the Radon measure $ \lambda $ on $\Lambda $, equipped with the usual inner-product. Recall that an integral operator $ K : L^2(E, \lambda) \to L^2(E, \lambda)$ is said to have kernel $ k : E^2 \to \mathbf C $ if $K$ is written
\[
    K : f \mapsto  \int k(x, y) f(y) \mathrm d\lambda(y).
\]
If $k$ is in $L^2(E^2, \lambda^{\otimes 2})$, the operator $K$ is a linear bounded operator on $L^2(E, \lambda)$.

\begin{definition} 
    Let now be a compact subset $ \Lambda \subset E $, the restriction of the integral operator $K$ to $\Lambda $ is the integral operator
    \[
        K^\Lambda : f \in L^2(\Lambda, \lambda) \to \int_\Lambda k(x,y) f(y) \mathrm d \lambda (y) \in L^2(\Lambda, \lambda).
    \]
    $ K^\Lambda $ is then a compact operator.
\end{definition}

\begin{definition} 
    The operator $K$ is said to be Hermitian or self-adjoint if its kernel $k$ verifies $ k(x, y) = \overline{k(y, x)}$ for $\lambda^{\otimes 2}$-almost every $(x, y) \in E^2$.
\end{definition}

Equivalently, this means that the integral operators $K^\Lambda$ are self-adjoint for any compact set $\Lambda \subset E$. If $K^\Lambda$ is self-adjoint, by the spectral theorem for self-adjoint and compact operators we have that $L^2(\Lambda, \lambda)$ has an orthonormal basis $(\phi_j^\Lambda)_{j \ge 0}$ of eigenfunctions of $K^\Lambda$. The corresponding eigenvalues $(\lambda_j^\Lambda)_{j\geqslant 0}$ have finite multiplicity (except possibly the zero eigenvalue) and the only possible accumulation point of the eigenvalues is zero. In that case, Mercer's theorem indicates that the kernel $k^\Lambda$ of $K^\Lambda$ can be written
\[ 
    k^\Lambda(x, y) = \sum_{n \geqslant 0} \lambda_n^\Lambda \phi_n^\Lambda(x) \overline{\phi_n^\Lambda(y)},
\]
where the $(\phi_k^\Lambda)_{k \ge 0}$ form a Hilbert basis of $L^2(E, \lambda)$ composed of eigenfunctions of $K^\Lambda$.

Recall that $K$ is positive if its spectrum is included in $\mathbb{R}^+$, and is of trace-class if $$ \sum_{n=1}^\infty |\lambda_n| < \infty,$$ its trace is then $\mathrm{Tr}(K) := \sum_{n=1}^\infty \lambda_n$.

\begin{definition}
    If $K^\Lambda$ is of trace-class for all compacts $\Lambda$, $K$ is said to be locally trace-class.
\end{definition}

\begin{hypothesis}\label{hyp:main-hypothesis}
    Throughout this paper, our kernels will be self-adjoint, locally trace class, with spectrum contained in $[0, 1]$.
\end{hypothesis}

We refer to \cite{ConwayOpTheory2000} and \cite{ConwayFunctAnalysis2019} for further developments on these notions.

\begin{definition} 
    A locally finite and simple point process on $E$ is a determinantal point process if its correlation functions with respect to the reference Radon measure $\lambda$ on $E$ exist and are of the form
    \[
        \rho_n(x_1, \dots, x_n) = \det(k(x_i, x_j))_{1 \le i,j \le n},
    \] where $k$ satisfies Hypothesis~\ref{hyp:main-hypothesis}.
\end{definition}

The dynamics of DPPs are described by the following fundamental theorem.

\begin{theorem}\label{thm:main-dpp-thm}
Under Hypothesis~\ref{hyp:main-hypothesis}, consider Mercer's decomposition of the kernel $k$ of the determinantal point process $\eta$:
\[
k(x,y) = \sum_{k=1}^n \lambda_k \phi_k(x) \overline{\phi_k(y)}.
\]
Here, the eigenvalues are all in $[0,1]$ and can all be chosen in $(0,1]$ ; $n$ is equal to the rank of $K$, which can be either finite or infinite. The $(\phi_n)$ form a Hilbert basis is the space $L^2(E, \lambda)$. Consider then a sequence of independent Bernoulli random variables $(B_k)_{1 \le k \le n}$, and consider the random kernel
\[
k_B(x, y) = \sum_{k=1}^n B_k \phi_k(x) \overline{\phi_k(y)},
\]
Then the point process $\eta_B$ with (random) kernel $k_B$ has the same law as that of $k_B$:
\[
\eta \overset{\text{Law}}{=} \eta_B.
\]
\end{theorem}

\begin{corollary}
If the $n$ above is finite, then the point process has a finite number of points almost surely. By independence, the Borel-Cantelli lemmas imply that the number of points is almost surely finite if and only if $K$ is of trace class. Since the integral operators that we will consider will all be locally of trace-class, the restrictions of our DPPs to compact subsets will all have a finite number of points almost surely.
\end{corollary}

For more details, we refer to \cite{Hough2006} and \cite{ShiraiTakahashi2003}.

\begin{definition}
    The Bergman determinantal point process is the determinantal point process on the open unit disc centered at the origin of the complex plane with kernel
\[
k(x,y) = \frac{1}{\pi} \frac{1}{(1 - x\bar{y})^2}.
\]
\end{definition}
See \cite{ZerosGaussianPowerSeries} and \cite{Hough2009ZerosGAF} for a thorough study.

The key motivation of this article is the algorithm that was introduced and studied in \cite{DecreusefondVergne2015}, \cite{DecreusefondMoroz2021}. It is recalled down below. Its goal is to simulate a DPP restricted to a ball centered at the origin and of radius $R$. It assumes that the set
\[
I = \{n \ge 0, B_n = 1\},
\]
of "active" Bernoulli random variables has been computed, and computes the positions of the points $(X_k)_{k \in I}$. Recall that this makes sense because restricting a DPP to a compact reduces the number of points to almost surely finite.

\begin{algorithm}
    \caption{Sampling of the locations of the points given the set $I$ of active Bernoulli random variables}
    \begin{algorithmic}[1]
    \State \textbf{Input :} $R, I$
    \State \textbf{Output :} The positions of the points $(X_i)_{i \in I}$
    \State Let $\varphi_I^R(x) = (\varphi_i^R(x), i \in I)$
    \State Draw $X_1$ from the distribution with density $\|\varphi_I^R(x)\|_{\mathbf{C}^{|I|}}^2 / |I|$
    \State $e_1 \leftarrow \frac{\varphi_I^R(X_1)}{\|\varphi_I^R(X_1)\|_{\mathbf{C}^{|I|}}}$
    \For{$i \leftarrow 2$ to $|I|$}
        \State Draw $X_i$ from the distribution with density
        \[ p_i(x) = \frac{1}{|I| - i + 1} \left( \|\varphi_I^R(x)\|_{\mathbf{C}^{|I|}}^2 - \sum_{k=1}^{i-1} |\langle e_k, \varphi_I^R(x) \rangle|^2 \right) \]
        \State $u_i \leftarrow \varphi_I^R(X_i) - \sum_{k=1}^{i-1} \langle e_k, \varphi_I^R(X_i) \rangle e_k$
        \State $e_i \leftarrow \frac{u_i}{\|u_i\|_{\mathbf{C}^{|I|}}}$
    \EndFor
    \end{algorithmic}
\end{algorithm}
  
Since simulating a countable infinity of Bernoulli random variables is unfeasible, we introduce
the truncation to $N$ eigenvalues of the restricted DPP, which is defined by its kernel
\[
k_N^\Lambda = \sum_{k=0}^{N-1} \lambda_n^\Lambda \phi_n^\Lambda(x) \overline{\phi_n^\Lambda(y)},
\]
it is the only one that's numerically feasible.

\section{The Bergman DPP on a disc}

The big question is now to know to what extent this truncation greatly affects the simulation of the DPP law (or if the error is only a minor issue). We provide the theoretical results that answer the last question in the negative for the Bergman DPP (and partially for general DPPs). This is the question that was asked in \cite{Decreusefond2016}.

\begin{notation}
    Throughout this paper and unless expressly stated otherwise, $R$ and $r$ denote real
numbers in the interval $(0, 1)$ such that $r < R$.
\end{notation}

Let us first contruct the restriction of the Bergman DPP to a compact ball centered at the origin
with radius $R$.

\begin{theorem} Denote $\mathcal{B}(0, R)$ the compact ball centered at 0 with radius $R$. Mercer's decomposition of the kernel $k^R(x,y)$ of the Bergman determinantal point process restricted to
$\mathcal{B}(0, R)$ is
\[
k^R(x,y) = \sum_{n \ge 0} \lambda_n^R \phi_n^R(x) \overline{\phi_n^R(y)},
\]
where the eigenvalues are
\[
\lambda_k^R = R^{2k+2},
\]
and the eigenfunctions
\[
\phi_k^R : x \mapsto \sqrt{\frac{k+1}{\pi}} \frac{1}{R^{k+1}} x^k.
\]
\end{theorem}

\begin{proof} 
    
We have, for the (original) Bergman kernel
\[
k(x,y) = \frac{1}{\pi} \frac{1}{(1 - x\bar{y})^2},
\]
using $\frac{1  }{(1-u)^2} = \sum_{k \ge 0} (k+1)u^k$ for all $u \in \mathcal{B}(0,R)$, we have
\begin{align*}
\sum_{k \ge 0} \frac{k+1}{\pi} x^k \bar{y}^k &= \sum_{k \ge 0} R^{2k+2} \frac{k+1}{\pi} \frac{x^k}{R^{k+1}} \frac{\bar{y}^k}{R^{k+1}} \\
&= \sum_{n \ge 0} \lambda_n^R \phi_n^R(x) \overline{\phi_n^R(y)},
\end{align*}
taking $\phi_n^R$ and $\lambda_n^R$ as in the theorem statement. Because we are restricting onto $L^2(\mathcal{B}(0,R), \lambda)$, we ought to consider the inner product on this space and not $L^2(\mathcal{B}(0,1), \lambda)$. We have
\[
\langle \phi_n^R, \phi_m^R \rangle_{L^2(\mathcal{B}(0,R))} = \sqrt{\frac{(n+1)(m+1)}{\pi}} \frac{1}{R^{n+m+2}} \int_{\mathcal{B}(0,R)} z^n \bar{z}^m \mathrm dz.
\]
However,
\[
\int_{\mathcal{B}(0,R)} z^n \bar{z}^m dz = \int_0^R \int_0^{2\pi} r^{n+m} e^{i\theta(n-m)} r \mathrm dr d\theta = 2\pi \delta_{n,m} \int_0^R r^{n+m+1} \mathrm dr = 2\pi \delta_{n,m} \frac{R^{2n+2}}{2n+2}.
\]

Consequently,
\[
\langle \phi_n^R, \phi_m^R \rangle_{L^2(\mathcal{B}(0,R))} = \delta_{n,m} \frac{\sqrt{(n+1)(m+1)}}{\pi} \frac{1}{R^{n+m+2}} 2\pi \frac{R^{2n+2}}{2n+2} = \delta_{n,m}.
\]
This prooves that this family is an orthonormal family of eigenfunctions of the integral operator $K^R$ associated to $k^R$, and we have hence found its Mercer decomposition.
\end{proof}

\begin{remark} 
    The fact that the eigenvalues and eigenfunctions of the Bergman DPP are explicitly computable makes it very manipulatable. Let us mention that restricting any DPP on any compact hardly ever gives birth to such well-behaving results. The only DPP that was ever discovered to enjoy such properties if the Ginibre DPP as shown in \cite{DecreusefondMoroz2021}.
\end{remark}

We will now make use of Optimal transport tools and show that, when choosing the right number of points (which we will exhibit), the law induced by the truncated version of this kernel is close (in the Wasserstein sense) to the non-truncated restriction.

Let $X$ and $Y$ be two Polish spaces. Let $\mu$ and $\nu$ be probability measures on $X$ and $Y$ respectively. Denote $\Pi(\mu, \nu)$ the set of probability measures on $X \times Y$, the first marginal of which is $\mu$ and the second $\nu$. If $c$ is a lower semi-continuous function from $X \times Y$ to $\mathbb{R}^+$, the Monge-Kantorovitch problem asks to find
\[
\inf_{\gamma \in \Pi(\mu,\nu)} \int_{X \times Y} c(x,y) \mathrm d\gamma(x,y).
\]
We refer to \cite{Villani2021} and \cite{Villani2009} for proper introductions to the topic of Optimal transport.

\begin{definition} Observing that the cardinality of the symmetric difference $d(\xi, \zeta) = |\xi \Delta \zeta|$ induces a distance on the set of configurations on a subset of $\mathbf{C}$, the Kantorovitch-Rubinstein distance is the Wasserstein-1 distance induced by $d$, that is
\[
\mathcal{W}_{KR}(\mu, \nu) = \inf_{\substack{\text{law}(\xi)=\mu \\ \text{law}(\zeta)=\nu}} \mathbf{E}(|\xi \Delta \zeta|) = \inf_{\substack{\text{law}(\xi)=\mu \\ \text{law}(\zeta)=\nu}} \mathbf{E}(d(\xi, \zeta)).
\]
It is a distance on the set of point process laws. See [5] for a thorough study of Wasserstein distances on configuration spaces. The following result exhibits a good compromise as to the number of points the Bergman ought to be truncated to.
\end{definition}

\begin{theorem}\label{thm:main-analog}
Let
\[
N_R := \sum_{n=0}^\infty R^{2n+2} = \frac{R^2}{(1-R)(1+R)}.
\]
Denote $\mathfrak{S}^R$ the law of the restricted Bergman to a compact ball of radius $R$ centered at 0 and $\mathfrak{S}_\alpha^R$ the law of its truncation to $\alpha$ eigenvalues.

If we truncate it to $\beta N_R$ eigenvalues, we have
\begin{equation}
\mathcal{W}_{KR}(\mathfrak{S}^R, \mathfrak{S}_{\beta N_R}^R) \leqslant N_R e^{-2\beta g(R)}
\end{equation}
where
\[
g(R) = \frac{R^2}{1+R}.
\]
\end{theorem}

\begin{proof}

Consider the coupling of $(\xi_{\beta N_R}^R, \xi^R)$ such that the first random variable is a subset of the second one consisting in the point whose indexes are $n \le \beta N_R$ in Mercer's decomposition. We then have
\[
\mathcal{W}_{KR}(\mathfrak{S}^R, \mathfrak{S}_{\beta N_R}^R) \leqslant \sum_{k=\beta N_R+1}^\infty R^{2k+2}.
\]
Introducing $\varepsilon = 1 - R$,
\begin{align*}
\mathcal{W}_{KR}(\mathfrak{S}^R, \mathfrak{S}_{\beta N_R}^R) &\leqslant \frac{(1-\varepsilon)^2}{\varepsilon(2-\varepsilon)} (1-\varepsilon)^{2\beta N_R} \\
&\leqslant \frac{(1-\varepsilon)^2}{\varepsilon(2-\varepsilon)} e^{2\beta N_R \log(1-\varepsilon)} \\
&\leqslant \frac{(1-\varepsilon)^2}{\varepsilon(2-\varepsilon)} \exp \left( 2\beta \frac{(1-\varepsilon)^2}{\varepsilon(2-\varepsilon)} \log(1-\varepsilon) \right) \\
&\leqslant \frac{(1-\varepsilon)^2}{\varepsilon(2-\varepsilon)} \exp \left( -2\beta \frac{(1-\varepsilon)^2}{2-\varepsilon} \right).
\end{align*}
The proof is complete.
\end{proof}

\begin{remark}
    As a corollary of the previous proof, the two point processes coincide with high probability. The previous proof shows that :
\end{remark}

\begin{proposition}\label{prop:analog-corollary}
We have
\[
\mathbf{P}(\mathfrak{S}^R \neq \mathfrak{S}_{\beta N_R}^R) \leqslant N_R e^{-2\beta \frac{R^2}{1+R}}.
\]
\end{proposition}

\begin{proof}
Introducing the Bernoulli random variables from Theorem~\ref{thm:main-dpp-thm}, we have
\[
\mathbf{P}(\mathfrak{S}^R \neq \mathfrak{S}_{\beta N_R}^R) \leqslant \mathbf{P}(\exists k > \beta N_R, B_k = 1) \leqslant \sum_{k > \beta N_R} \mathbf{P}(B_k = 1) = \sum_{k > \beta N_R} \lambda_k^R \leqslant N_R e^{-2\beta \frac{R^2}{1+R}}
\]
as wanted.
\end{proof}

\begin{remark}\label{rem:truncation}
    In \cite{DecreusefondMoroz2021}, the authors have chosen to truncate the restricted Ginibre DPP to $N_R = R^2$ eigenvalues as seen in result \eqref{eq:main-result-ginibre} that shows that the truncation error is exponentially small in the deviation $c$ from $N_R$. Though the authors of \cite{DecreusefondMoroz2021} have described this as a "well-known \textit{observation}", it is interesting to observe that this $R^2$ is by no means random and can be theoretically forecasted. This is because turns out to correspond exactly to the expectation of the number of points of the Ginibre DPP. See for yourself : though very interesting to be pointed out, the following proposition is to be found nowhere in the litterature. It motivates to truncate to the expectation and to study the deviation that we have written in Theorem~\ref{thm:main-analog}.
\end{remark}

\begin{proposition} The expected number of points $\mathbf{E}[|\mathfrak{S}_R^G|]$ that will come out from the restricted Ginibre DPP to $\mathcal{B}(0,R)$ is exactly $R^2$ (here, $R$ can be any positive real number).
\end{proposition}

\begin{proof}

According to the fundamental theorem of DPPs (Theorem~\ref{thm:main-dpp-thm} in this paper), this expectation is $\sum\limits_{n \ge 0} \lambda_n^R$ where $\lambda_n^R$ denotes the $n$-th eigenvalue of the restricted Ginibre integral operator.

We have (see \cite{DecreusefondMoroz2021}) $\lambda_n^R = \frac{\gamma(n+1, R^2)}{n!}$ where $\gamma$ stands for the incomplete gamma function $\gamma(n,x) = \int_0^x t^{n-1}e^{-t} \mathrm dt$.

We have
\[
\sum_{n=0}^\infty \lambda_n^R = \sum_{n=0}^\infty \frac{1}{n!} \gamma(n+1, R^2) = \sum_{n=0}^\infty \frac{1}{n!} \int_0^{R^2} t^{n+1-1} e^{-t} \mathrm dt = \int_0^{R^2} \sum_{n=0}^\infty \frac{1}{n!} t^n e^{-t} \mathrm dt = \int_0^{R^2} \mathrm dt = R^2,
\]
as wanted.
\end{proof}

\begin{remark} The behavior of the rational function $\frac{R^2}{1+R}$ inside the exponential function is to be nicely bounded in the neighborhood of $R \to 1$. In other words, the bound (the deviation) is indeed exponential (exponentially small) in $\beta$. For instance, for $R \in [0.9, 1]$, we have $g(R) = \frac{R^2}{1+R} \ge 0.42$ so we get $\leqslant \dots \exp(-0.98\beta)$. For $R \in [0.99, 1]$, $g(R) \ge 0.493$, so $\leqslant \dots \exp(-0.986\beta)$, and so on. (Since $g(R)$ is continuous and vanishes at 0, it cannot be bounded from below on $[0,1]$ by a constant $K_0 > 0$. This is why here have to consider subintervals that do not touch 0 and see the bounds on those).
\end{remark}

\begin{remark} We have $N_R \xrightarrow[R \to 1^-]{} +\infty$. This comforts in thinking that the number of points is well-chosen, since as $R \to 1^-$, we seem to find a restriction that comes closer and closer to the original Bergman point process.
\end{remark}

The following results further confirm this idea.

\begin{proposition} Denoting $\mathfrak{S}^R$ the law of the restricted Bergman to a compact ball of radius $R$ centered at the origin and $\mathfrak{S}_N$ the law of its truncation to $N$ eigenvalues, we have
\[
\mathfrak{S}_N^R \xrightarrow[N \to \infty]{} \mathfrak{S}^R,
\]
in distribution.
\end{proposition}

\begin{proof}

The goal is to reduce this result to Theorem~\ref{thm:cvg-in-law-dpp}. Denote $\|\cdot\|_\infty$ the norm of uniform convergence. Let $\Lambda \subset \mathcal{B}(0,R)$ be compact. Denote $k_N^R$ the truncated kernel to $N$ eigenvalues and $k^R$ the asymptotic one. We have
\[
\|k_N^R - k^R\|_{\infty, \Lambda} = \left\| \frac{1}{\pi} \sum_{k=N}^\infty (k+1) x^k \bar{y}^k \right\|_{\infty, \Lambda}.
\]
Observe that $\Lambda \subset \mathcal{B}(0,R)$ implies that
\begin{align*}
\|k_{N}^R - k^R\|_{\infty, \Lambda} &\leqslant \frac{1}{\pi} \sum_{k=N}^\infty (k+1) \|x^k \bar{y}^k\|_{\infty, \Lambda} \\
&\leqslant \frac{1}{\pi} \sum_{k=N}^\infty (k+1) R^{2k} \xrightarrow[N \to \infty]{} 0.
\end{align*}
The proof is complete.
\end{proof}

\begin{proposition} Denote $\mathfrak{S}$ the original Bergman DPP and consider truncations $ \mathfrak{S}_{N_R}^R $ to $ N_R $ eigenvalues of its restriction to the disc or radius $R$. Provided that $ N_R \xrightarrow[R \to 1^-]{} +\infty $, we have
\[
\mathfrak{S}_{N_R}^R \xrightarrow[R \to 1^-]{} \mathfrak{S},
\]
in distribution.
\end{proposition}

\begin{proof}
Denote $k_{N_R}^R$ the restricted-truncated (to $N_R$ eigenvalues) kernel and $k$ the asymptotic one. Consider a compact $\Lambda \subset D(0,1)$. Recall that the Bergman kernel $k(x,y) = \frac{1}{\pi}\frac{1}{(1-x\bar{y})^2}$ is defined on the open disc $D(0,1)$ (this relation would not make sense on some points of the boundary).

Since $D(0,1)$ is open, denoting $m := \max_{z \in \Lambda} |z|$, we have $m < 1$. So, we have
\begin{align*}
\|k_{N_R}^R - k\|_{\infty, \Lambda} &= \left\| \frac{1}{\pi} \sum_{k=N_R}^\infty (k+1) x^k \bar{y}^k \right\|_{\infty, \Lambda} \\
&\leqslant \frac{1}{\pi} \sum_{k=N_R}^\infty (k+1) \|x^k \bar{y}^k\|_{\infty, \Lambda} \\
&\leqslant \frac{1}{\pi} \sum_{k=N_R}^\infty (k+1) m^{2k} \xrightarrow[R \to 1^-]{} 0
\end{align*}
since $ N_R \xrightarrow[R \to 1^-]{} +\infty $.
\end{proof}

\begin{theorem}\label{thm:cvg-in-law-dpp} (Proposition 3.10, \cite{ShiraiTakahashi2003})
Let $(K_n)_{n \ge 0}$ be integral operators with nonnegative continuous kernels $k_n(x,y)$. Assume that $K_n$ are bounded hermitian locally trace class integral operators. Assume that $(k_n)_{n \ge 0}$ converges to a kernel $k$ uniformly on each compact as $n \to \infty$. Then, the kernel $k$ defines an integral operator $K$ that is also, bounded, hermitian, and locally trace class. Furthermore, the determinantal measures, i.e. the DPP probability laws $\mu_n$ associated to the DPP induced by the integral operators $K_n$, weakly converge to the determinantal probabilty measure $\mu$ induced by $K$.
\end{theorem}

\begin{remark} 
A look at the proof of Theorem~\ref{thm:main-analog} hints that
\begin{equation}
\mathcal W_{KR} ( \mathfrak S^R_{N}, \mathfrak S^R ) \xrightarrow[N \to \infty]{} 0    
\end{equation}
Still motivated by the question of the convergence for our variants of the Bergman DPP, one could ask whether or not some convergences between some of our variants hold in the Wasserstein sense ; in other words, whether (which of) the following claims hold :
\begin{equation}
\mathcal W_{KR} ( \mathfrak S^R, \mathfrak S) \xrightarrow[R \to 1^-]{} 0
\end{equation}
\begin{equation}\label{middle-conjecture}
\mathcal W_{KR} ( \mathfrak S^R_{N_R}, \mathfrak S^R ) \xrightarrow[R \to 1^-]{} 0
\end{equation}
\begin{equation}
\mathcal W_{KR} ( \mathfrak S^R_{N_R}, \mathfrak S ) \xrightarrow[R \to 1^-]{} 0
\end{equation}
And more generally, under what assumption on $ N_R $, besides going to infinity as $ R \to 1^- $, do these claims hold or not. 

We provide an element of answer to the very last question for \eqref{middle-conjecture} in the form of a sufficient condition in the following result. For the rest, these questions are open. Note that two of these claims would imply the third one via the triangular inequality for the distance $ \mathcal W_{KR} $.
\end{remark}

\begin{proposition} Let $ \varepsilon = 1 - R $. If 
\begin{equation}\label{strong-assumption-n}
N_\varepsilon \underset{\varepsilon \to 0^+}{\sim} \frac{1}{\varepsilon^{1+\delta}},
\end{equation}
then \eqref{middle-conjecture} holds.
\end{proposition}

More generally, \eqref{middle-conjecture} holds if 
\begin{equation}\label{weak-assumption-n}
2N_\varepsilon  \log(1-\varepsilon) - \log(\varepsilon) \xrightarrow[\varepsilon \to 0^+]{} - \infty .
\end{equation}

\begin{proof}

Assuming \eqref{strong-assumption-n}, we have
\begin{align*}
    2N_\varepsilon \log(1-\varepsilon) &\leqslant - 2N_\varepsilon \varepsilon \\
                                       &\underset{\varepsilon \to 0^+}{\sim} - 2 \varepsilon^{-\delta}
\end{align*}
So, knowing how logs compare to power functions at zero, \eqref{weak-assumption-n} holds and it suffices to show \eqref{middle-conjecture} assuming \eqref{weak-assumption-n}.

We have
\begin{align*}
    \mathcal{W}_{KR}(\mathfrak{S}^{1-\varepsilon}, \mathfrak{S}_{N_\varepsilon }^{1-\varepsilon}) 
    &\leqslant \frac{(1-\varepsilon)^2}{\varepsilon(2-\varepsilon)} (1-\varepsilon)^{2 N_R} \\
    &\leqslant \frac{(1-\varepsilon)^2}{(2-\varepsilon)} e^{2 N_R \log(1-\varepsilon) - \log(\varepsilon)}
\end{align*}
We have the result.
\end{proof}

\begin{remark} In particular, considering $N_R$ from Theorem~\ref{thm:main-analog}, convergence holds in the Wasserstein sense holds replacing $ N_R$ by $ N_R^{1+\delta} $.
\end{remark}

\section{The Bergman DPP on other regions}

We now consider a variant of the previous restriction. A sight at Figure~\ref{fig:main-fig} in the introduction shows that the points in the Bergman DPP are very highly repelled from the center, and very intensely concentrated at the border of the disc they're living in. The following was shown in \cite{ZerosGaussianPowerSeries} :

\begin{theorem}\label{thm:moduli}
The law of the set of the moduli $\{|X_k|, k \ge 1\}$ of the points that come out from the Bergman determinantal point process is exactly the law of the set
\[
\{U_k^{1/(2k)}, k \ge 1\}
\]
where $(U_k)_{k \ge 1}$ is a sequence of independent, uniform in $[0,1]$ random variables.
\end{theorem}

\begin{remark} (Conjecture)

According to observations (see Figure~\ref{fig:main-fig}), we may conjecture that, conditionnally on the set $I = \{n \ge 0, B_n = 1\}$ of active Bernoulli random variables, the law of the set of moduli $\{|X_k|, k \in I\}$ of the points for the \textit{restricted} Bergman determinantal point process is exactly the law of the set
\[
\{U_k^{1/(2k)}, k \in I\}
\]
where $(U_k)_{k \ge 1}$ is a sequence of independent, uniform in $[0, R]$ random variables.
\end{remark}

We have the following.

\begin{proposition} Let $m := \min_{1 \le k \le n} (|X_k|)$ denote the smallest radius among the $n$ first points of the Bergman point process (as always, "first" in the sense of Mercer's decomposition).

We have, for all $x \in [0,1]$
\[
\mathbf{P}(m \leqslant x) = 1 - \prod_{k=1}^n (1 - x^{2k}).
\]
\end{proposition}

\begin{proof}

We have
\[
\mathbf{P}(U_k^{1/(2k)} \ge x) = \mathbf{P}(U_k \ge x^{2k}) = 1 - x^{2k}.
\]
According to Theorem~\ref{thm:moduli}, we have
\[
\mathbf{P}(m \ge x) = \prod_{k=1}^n \mathbf{P}(U_k^{1/(2k)} \ge x) = \prod_{k=1}^n (1 - x^{2k}).
\]
The proof is complete.
\end{proof}

\begin{remark} As a corollary, we have
\[
\mathbf{P}(m \leqslant x) \underset{x \to 0}{\sim} x^2
\]
because the polynomial involved is even and vanishes at zero. This suggests that restricting to a annulus instead of a ball could be enough. Though it is to be mentionned that this has not yet been implemented in \cite{MorozSoftware}, we construct this restriction.
\end{remark}

\begin{theorem} Denote $T(r,R)$ the compact annulus centered at 0 with inner radius $r$ and outer radius $R$. Mercer's decomposition of the kernel $k_{r,R}(x,y)$ of the Bergman determinantal point process restricted to $T(r,R)$ is
\[
k_{r,R}(x,y) = \sum_{n \ge 0} \lambda_n^{r,R} \phi_n^{r,R}(x) \overline{\phi_n^{r,R}(y)}
\]
where the eigenvalues are
\[
\lambda_k^R = R^{2k+2} - r^{2k+2}
\]
and the eigenfunctions
\[
\phi_k^R: x \mapsto \sqrt{\frac{k+1}{\pi(R^{2k+2} - r^{2k+1})}} x^k
\]
\end{theorem}

\begin{proof}
We have, for the (original) Bergman kernel
\[
k(x,y) = \frac{1}{\pi} \frac{1}{(1-x\bar{y})^2},
\]
using $\frac{1}{(1-u)^2} = \sum_{k \ge 0} (k+1)u^k$ for all $u \in T(r,R)$, we have
\[
\sum_{k \ge 0} \frac{k+1}{\pi} x^k \bar{y}^k = \sum_{n \ge 0} \lambda_n^{r,R} \phi_n^{r,R}(x) \overline{\phi_n^{r,R}(y)}
\]
taking $\phi_n^{r,R}$ and $\lambda_n^{r,R}$ as in the theorem statement. We have
\[
\langle \phi_n^{r,R}, \phi_m^{r,R} \rangle_{L^2(T(r,R))} = \frac{1}{\pi} \sqrt{\frac{(n+1)(m+1)}{\lambda_n^{r,R} \lambda_m^{r,R}}} \int_{T(r,R)} z^n \bar{z}^m \mathrm dz
\]
However,
\[
\int_{T(r,R)} z^n \bar{z}^m \mathrm dz = \int_r^R \int_0^{2\pi} r^{n+m} e^{i\theta(n-m)} \mathrm d\theta rdr = 2\pi \delta_{n,m} \int_r^R r^{n+m+1} \mathrm dr = 2\pi \delta_{n,m} \left[ \frac{R^{2n+2}}{2n+2} - \frac{r^{2n+2}}{2n+2} \right]
\]
Consequently,
\[
\langle \phi_n^{r,R}, \phi_m^{r,R} \rangle_{L^2(T(r,R))} = \delta_{n,m}.
\]
The proof is thus complete.
\end{proof}

\begin{definition} If $A$ is a Borel subset of $[0,1]$, in the sequel, we will denote 
\[
Z(A) = \{z \in \mathbb{C}, |z| \in A\}.
\]
\end{definition}

\begin{remark}\label{rem:general-eigenvalues}
The previous proof actually shows that it is also feasible to compute the restriction to a domain of the form $ Z(A) $, yielding eigenvalues
\[
\lambda_n^A = (2n+2)\int_A r^{2n+1} \mathrm dr,
\] 
with the same corresponding eigenfunctions up to a normalization factor. The proof is essentially identical replacing $ \int_r^R $ by $ \int_A $ and renormalizing the eigenfunctions.
\end{remark}

Not only is it interesting itself, the following proposition will be necessary for the proof of next theorem.

\begin{proposition} Let $ A $ and $B$ be two closed subsets of $ [0,1] $ such that $ A \subset B $. Consider the Bergman DPP $ \mathfrak S^A $ (resp. $ \mathfrak S^B $) restricted to $Z(A)$ (resp. to $Z(B)$). Denote $ \lambda^A_n $ (resp. $ \lambda^B_n $) its $n$-th eigenvalue (the (only) one associated to the eigenfunction which is a monomial of degree $n$)

Then we have 
\[
    \forall n \in \mathbf N, \qquad \lambda_n^A \leqslant \lambda_n^B
\]
\end{proposition}

\begin{proof}

This is an immediate corollary of Remark~\ref{rem:general-eigenvalues}.

\end{proof}

\begin{remark}\label{rem:adherence-property}
    Here we have thus constructed restrictions of the Bergman DPP to domains of the form $ Z([0,R]) $ and $ Z([r,R]) $. These restrictions are of interest for the simulation because as opposed to the original Bergman process, they exhibit a finite number of points almost surely. However, the main « loss » induced by such restrictions, clearly, is that we lack the possibility to simulate points in the annulus $ Z([1-\varepsilon,1]) $. But, given the observations of Figure~\ref{fig:main-fig}, it seems that it is really a shame since the points seem to be extremely concentrated on the border. A natural idea would be to restrict the Bergman DPP to a region of the form $ Z(A) $ such that $ [1-\varepsilon, 1] \subset A$ instead of $ \mathcal B(0,R) = Z([0,R]) $ for instance. Recalling that simulating an infinite number of points is unfeasible, this approach is quite doomed to failure as we show in the next result.
\end{remark}

\begin{theorem} The restriction of the Bergman DPP to any region of the form $ Z(A) $ ($ A \subset [0,1] $) that contains $ T(1-\varepsilon,1) = Z([1-\varepsilon,1]) $ exhibits an infinite number of points almost surely.
\end{theorem}

\begin{proof}
As explained in Remark~\ref{rem:general-eigenvalues}, such restriction yields the eigenvalues
\[
\lambda_n^A = (2n+2)\int_A r^{2n+1} \mathrm d r .
\]
However, we then have 
\[
\lambda_n^A \geqslant (2n+2) \int_{[1-\varepsilon, 1]} r^{2n+1} \mathrm d r = 1 - (1-\varepsilon)^{2n+2}.
\]
The series $ \displaystyle \sum_{n\ge 0} (1-\varepsilon)^{2n+2} $ is convergent, hence the sum $ \displaystyle \sum_{n \ge 0} \lambda_n^A $ is not. As a result, the corresponding integral operator is not trace class, and so the corresponding DPP exhibits an infinite number of points almost surely.
\end{proof}

\begin{remark}
    So far we have considered the families of regions $ \mathcal F_1 = (Z([1-\varepsilon, 1]))_{\varepsilon \in (0,1]} $, $ \mathcal F_2 = (Z([0,R]))_{R \in [0,1)} $ and $ \mathcal F_3 = (Z([r,R]))_{\substack{r,R \in [0,1) \\ r < R }} $. Although we have just shown that the regions of $ \mathcal F_1 $ are not a good choice for feasible simulation, this family satisfies the following property :
\end{remark}

\begin{center}\label{property-a}
    \hfill For all $ X \in \mathcal F$, for all $ \delta > 0 $, the (topological) interior of $Z([1-\delta, 1]) \cap X$ is nonempty. \hfill (A)
\end{center} 

In other words, (the interior of) the regions $X \in \mathcal F_1$ is not included in some $ Z([0,R]) $ (with $ R < 1 $) : its interior « touches » the unit circle (topologically). We have explained in Remark~\ref{rem:adherence-property} why this property matters. Note that as opposed to $\mathcal F_1$, the families $\mathcal F_2$ and $\mathcal F_3$ do not satisfy this property.

It is thus natural to ask the following question : does there exist a family of regions $ \mathcal F $, for which \hyperref[property-a]{(A)} holds and retricted to which the Bergman DPP exhibits a finite number of points almost surely ?

Recall that the Lebesgue measure $ \lambda(\mathcal B(0,1)) $ of the whole unit disc $ \mathcal B(0,1)$ is $ \pi $. Another interesting property verified by the three families of regions $ \mathcal F_1 $, $ \mathcal F_2 $ and $ \mathcal F_3 $ is the following :

\begin{center}\label{property-b}
    \hfill For all $ \delta > 0 $, there exists $ Z \in \mathcal F $ such that $ \lambda (Z) \geqslant \pi - \delta $. \hfill (B)
\end{center}

Intuitively, the most the region to which the Bergman DPP is restricted covers $ \mathcal B(0,1) $, the closest the approcimation should be. This is why property~\hyperref{property-b} matters too (although we do not here provide the probabilistic results that express this idea). 

So, the question is now : can we find a family of regions $ \mathcal F $, such that the Bergman DPP restricted to any $ Z \in \mathcal F $ exhibits a finite number of points almost surely, \textit{and} that satisfies both properties \hyperref[property-a]{(A)}, \hyperref[property-b]{(B)} ?

We now answer this question in the positive and explicitely contruct such family.

\begin{theorem}
Let $ (u_n)_{n\geqslant} $ be any sequence of positive numbers such that 
\[ 
    \sum_{n \ge 0} u_n < + \infty. 
\] 
Let $ (a_n)_{n \in \mathbf N} $ and $ (b_n)_{n \in \mathbf N} $ be two sequences valued in $ (0,1) $ with $ 0 < a_0 < b_0 < 1 $ and that satisfy the following recursive relation :
\[
\forall n \in \mathbf N, 
\left\{
\begin{aligned}
    a_{n+1} &\in (b_n, 1) \\
    b_{n+1} &= \sqrt{ \frac{a_{n+1}^2 + u_n(1-a_{n+1}^2)}{a_{n+1}^2 + (1+u_n)(1-a_{n+1}^2)} }
\end{aligned}
\right.
\]
Then the Bergman determinantal point process restricted to 
\[
    Z\left( \bigcup_{n \ge 0} [a_n, b_n]\right) 
\]
exhibits a finite number of points almost surely. Furthermore, choosing $ a_n $'s in the recurrence relation such that $ a_n \xrightarrow[n \to \infty ]{} 1 $ ensures that property (A) is satisfied, and further choosing $ b_0 \geqslant 1 - \delta < 1$ for $ \delta $'s approaching zero yields regions that ensure that property (B) is satisfied. 
\end{theorem}

\begin{proof}
First note that such sequences will be valued in $(0,1)$ because the first terms $ a_0, b_0$ are (must be) chosen in $ (0,1) $, and recursively, assuming that $ a_0, ..., a_n, b_0, ..., b_n $ are well-defined and in $(0,1)$ for some $ n \in \mathbf N$, we see that so are $ a_{n+1}, b_{n+1} $ by a sight at the recurrence relation and the fact that $ u_n > 0 $. 

Then, let us show that we have $ a_0 < b_0 < ... < a_n < b_n < ... \,$. First, $ a_0, b_0 $ are (must be) chosen to satisfy $ a_0 < b_0 $. Recursively, if, for some $ n \in \mathbf N $, we have $ a_0 < b_0 < ... < a_n < b_n $, then $ a_{n+1} > b_{n} $ by assumption, and it remains to show that $ b_{n+1} > a_{n+1} $.

We have
\[
    b_{n+1}^2 = \cfrac{a_{n+1}^2 + u_n(1-a_{n+1}^2)}{a_{n+1}^2 + (1+u_n)(1-a_{n+1}^2) } = \cfrac{ \cfrac{a_{n+1}^2}{1-a_{n+1}^2} + u_n}{ \cfrac{a_{n+1}^2}{1-a_{n+1}^2} + u_n + 1} 
\]
\[
    1 - b_{n+1}^2 = \cfrac{1}{\cfrac{a_{n+1}^2}{ 1-a_{n+1}^2} + u_n + 1}
\]
\[ 
    \cfrac{b_{n+1}^2}{1 - b_{n+1}^2} = \cfrac{a_{n+1}^2}{1-a_{n+1}^2} + u_n,
\]
Since $ u_n > 0 $ for all $ n \in \mathbf N $, we have
\[
    \cfrac{b_{n+1}^2}{1 - b_{n+1}^2} > \cfrac{a_{n+1}^2}{1-a_{n+1}^2},
\]
So, since $ u \mapsto \cfrac{u}{1+u} $ is strictly increasing on $ (0,1) $, 
\[
    \cfrac{ \cfrac{b_{n+1}^2}{1 - b_{n+1}^2}}{ 1+\cfrac{b_{n+1}^2}{1 - b_{n+1}^2}} > \cfrac{ \cfrac{a_{n+1}^2}{1-a_{n+1}^2}}{ 1+\cfrac{a_{n+1}^2}{1-a_{n+1}^2}}
\]
\[
    \cfrac{b_{n+1}^2}{1-b_{n+1}^2+b_{n+1}^2} > \cfrac{a_{n+1}^2}{1-a_{n+1}^2+a_{n+1}^2},
\]
From there, $ b_{n+1} > a_{n+1} $ follows.

The eigenvalues of the Bergman intergral operator retricted to 
\[
    A := Z\left( \bigcup_{n \ge 0} [a_n, b_n]\right),
\]
write as
\begin{align*}
    \lambda_n &= (2n+2) \int_A r^{2n+1} \mathrm d r \\
    &= \sum_{k \ge 0} b_k^{2n+2} - a_k^{2n+2},
\end{align*}
Recall that the Bergman restriction has a finite number of points almost surely if and only of the associated integral operator is trace-class. We have
\begin{align*}
    \sum_{n \ge 0} \lambda_n &= \sum_{n \ge 0} \sum_{k \ge 0} b_k^{2n+2} - a_k^{2n+2} \\
    &= \sum_{k \ge 0} \sum_{n \ge 0} b_k^{2n+2} - a_k^{2n+2} \\
    &= \sum_{k \ge 0} \frac{b_k^2}{1-b_k^2} - \frac{a_k^2}{1-a_k^2} \\
    &= \frac{b_0^2}{1-b_0^2} - \frac{a_0^2}{1-a_0^2} + \sum_{k \ge 1} \frac{b_k^2}{1-b_k^2} - \frac{a_k^2}{1-a_k^2} \\
    &= \frac{b_0^2}{1-b_0^2} - \frac{a_0^2}{1-a_0^2} + \sum_{k \ge 0} u_k \\
    &< +\infty.
\end{align*}
This prooves the main claim. The rest follows easily.
\end{proof}

\section{General deviations results for determinantal point processes}

As stated in Remark~\ref{rem:truncation}, our approach indicates that truncating restricted DPPs to a number of points equal to the expectation of their cardinality yields strong results such as Theorem~\ref{thm:main-analog} and Proposition~\ref{prop:analog-corollary}. We here further study the deviation of this cardinality through the following result. It is worthy to note that it is general as it applies not only to the Bergman kernel but to any determinantal point process.

\begin{theorem}
Let $\mathfrak S $ be a determinantal point process. Assume that its associated integral operator is trace-class. Denote $m$ the average number of points of $ \mathfrak S$. Then, $ m $ is finite, and we have for $ c \in (0,1) $,
\[
    \mathbf P( | \mathfrak S | \leqslant (1-c) m ) \leqslant \exp( -m(c +(1-c) \log(1-c) ) ),
\]
and for $ c > 0 $,
\[
    \mathbf P( | \mathfrak S | \geqslant (1+c) m ) \leqslant \exp( - m ((1+c) \log(1+c) -c)).
\]
\end{theorem}

\begin{remark}
    This is, to some extent, a generalization of Lemma 15 in \cite{DecreusefondMoroz2021}, to any DPP.
\end{remark}

\begin{proof}
Denote $B_n$ the $n$-th Bernoulli random variable (see Theorem~\ref{thm:main-dpp-thm}), we have $ B_n \sim \mathrm{Ber} (\lambda_n) $. Denote $ S=\sum\limits_{n=0}^\infty B_n$. Note that $m = \mathbf E(|\mathfrak S |) = \mathbf E(S)$. The associated integral operator is trace-class, so $m$ is finite. We have, for $ c \in (0,1) $ :
\begin{align*}
    \mathbf P( | \mathfrak S | \leqslant (1-c)m ) &= \mathbf P( S \leqslant (1-c)m ) \\
    &= \mathbf P\left[ (1-c)^{(1-c)m} \leqslant (1-c)^{S} \right] \\
    &\leqslant (1-c)^{-(1-c)m} \mathbf E[(1-c)^{S} ] \\
    &= (1-c)^{-(1-c)m} \mathbf E\left[ (1-c)^{\sum\limits_{n=0}^\infty B_n} \right] \\
    &= (1-c)^{-(1-c)m} \prod_{n=0}^\infty \mathbf E[ (1-c)^{B_n} ] \\
    &= (1-c)^{-(1-c)m} \prod_{n=0}^\infty (1-\lambda_n + \lambda_n(1-c)) \\
    &= (1-c)^{-(1-c)m} \prod_{n=0}^\infty (1-c\lambda_n) \\
\end{align*}
using $ 1+u \leqslant e^u $,
\begin{align*}
    &\leqslant (1-c)^{-(1-c)m} \prod_{n=0}^\infty e^{-c\lambda_n} \\
    &= (1-c)^{-(1-c)m} e^{-c \sum\limits_{n=0}^\infty \lambda_n} \\
    &= \exp(-(1-c)m \log(1-c) -c m) \\
    &= \exp(-m (c + (1-c) \log(1-c) )).
\end{align*}
This prooves the first part. For the second part,
\begin{align*}
    \mathbf P( | \mathfrak S | \geqslant (1+c)m ) &= \mathbf P( S \geqslant (1+c)m ) \\
    &= \mathbf P\left[ (1+c)^{S} \geqslant (1+c)^{(1+c)m} \right] \\
    &\leqslant (1+c)^{-(1+c)m} \mathbf E[(1+c)^{S} ] \\
    &= (1+c)^{-(1+c)m} \mathbf E\left[ (1+c)^{\sum\limits_{n=0}^\infty B_n} \right] \\
    &= (1+c)^{-(1+c)m} \prod_{n=0}^\infty \mathbf E[ (1+c)^{B_n} ] \\
    &= (1+c)^{-(1+c)m} \prod_{n=0}^\infty (1-\lambda_n + \lambda_n(1+c)) \\
    &= (1+c)^{-(1+c)m} \prod_{n=0}^\infty (1+c\lambda_n),
\end{align*}
using $ 1+u \leqslant e^u $,
\begin{align*}
    &\leqslant (1+c)^{-(1+c)m} \prod_{n=0}^\infty e^{c\lambda_n} \\
    &= \exp(-(1+c)m \log(1+c) +c m) \\
    &= \exp( - m ((1+c) \log(1+c) -c)). \\
\end{align*}
This prooves the second part.
\end{proof}

\newpage
\bibliographystyle{plain}
\bibliography{Article}

\end{document}